\documentclass[letterpaper, 10 pt, conference]{ieeeconf}  

\IEEEoverridecommandlockouts                              
\title{\LARGE \bf
Change Detection with the Kernel Cumulative Sum  Algorithm
}

\author{Thomas Flynn and Shinjae Yoo
  \thanks{This project was supported by the
    DOE Office of Science under grant ASCR KJ040201.}%
  \thanks{T. Flynn {(\tt\small tflynn@bnl.gov)} and S. Yoo {(\tt\small sjyoo@bnl.gov)} are with Computational Science Initiative,
  Brookhaven National Laboratory, Upton, NY 11973, United States.}%
}

\usepackage{amsmath}
\usepackage{amssymb}
\usepackage{framed}
\usepackage{color}
\usepackage{enumerate}
\usepackage{graphicx}
\usepackage[ruled,algosection]{algorithm2e}
\usepackage{booktabs}
\usepackage{multirow}

\newtheorem{thm}{Theorem}
\numberwithin{thm}{section}

\newtheorem{cor}[thm]{Corollary}
\newtheorem{prop}[thm]{Proposition}
\newtheorem{lem}[thm]{Lemma}

\newtheorem{example}[thm]{Example}

\newcommand{\reals}{\mathbb{R}}

\newcommand{\setin}{\subseteq}
\newcommand{\mc}{\mathcal}

\newcommand{\esssup}{\operatorname{esssup}}
\newcommand{\citep}{\cite}

\begin{document}

\maketitle
\begin{abstract}
  Online change detection involves monitoring a stream of  data for  changes in the statistical properties of incoming observations. A good change detector will detect any changes shortly after they occur, while raising few false alarms. Although there are algorithms with confirmed optimality properties for this task, they rely on the exact specifications of the relevant probability distributions and this limits their practicality. In this work we describe a kernel-based variant of the Cumulative Sum (CUSUM) change detection algorithm that can detect changes under less restrictive assumptions.
  Instead of using the likelihood ratio, which is a parametric quantity, the Kernel CUSUM (KCUSUM) algorithm compares incoming data with samples from a reference distribution using a statistic based on the Maximum Mean Discrepancy (MMD) non-parametric testing framework.
  The KCUSUM algorithm is applicable in settings where there is a large amount of background data available and it is desirable to detect a change away from this background setting. Exploiting the random-walk structure of the test statistic, we derive bounds on the performance of the algorithm,
including the expected delay and the average time to false alarm. 
\end{abstract}


\section{Introduction}
In this work we are interested in the problem of detecting abrupt changes in streams of data. This could mean detecting a change in the average value of the observations, or a change in variance, or, more generally, finding a change in any other distributional property. In particular we are interested in online change detection, where the algorithm should figure out a change has occurred soon after it happens, without waiting for the entire data stream to be observed. Some examples of where this is relevant include intrusion detection, industrial quality control, and others.

In cases where sufficient  prior knowledge of the change is available, there are known optimal algorithms for online change detection. If the probability distributions before and after the change are known, then the CUSUM procedure (shown in Algorithm \ref{algo:cusum}) is known to be optimal for an objective function that takes into account the magnitude of delays and the frequency of false alarms \citep{lorden1971,moustakides1986optimal}.
Aside from optimality, the CUSUM is also simple to program and has an intuitive interpretation in terms of maximum likelihood. However, there are many situations where the relevant probability distributions can not be modeled precisely, and it would be difficult to use the CUSUM in these cases.

Closely related to change detection is statistical hypothesis testing.
In particular, in this work we attempt to leverage tools that have been developed for the problem of two-sample testing and adapt them for use in change detection. Two-sample testing is a non-parametric hypothesis testing task where the goal is to determine if two data sets come from the same distribution. For this problem, an approach based on kernel embeddings has been developed, termed Maximum Mean Discrepancy (MMD) \citep{gretton2012kernel}.
In the MMD approach, two datasets are compared by computing the distance between the corresponding empirical measures, using a distance which is induced by a positive definite kernel function (See Section \ref{sect:mmd} for formal definitions.) Compared to other distance measures, MMD distances are appealing because they admit very simple unbiased estimators. Furthermore, being defined through kernels, the methods are not restricted to Euclidean datasets, and are applicable to hypothesis testing problems involving strings, graphs,  and other structured data \cite{meanembeddingreview}.

Motivated by prior work using MMD in hypothesis testing,
we introduce the Kernel Cumulative Sum (KCUSUM) algorithm
(Algorithm \ref{algo:twosampcd}).
Unlike the CUSUM,
the KCUSUM does not require exact specifications of the pre- and post-change distributions.
Instead,
it relies on a database of samples from the pre-change distribution,
and continuously compares incoming observations
with
samples from the database using a kernel function chosen by the user.
In this way, the KCUSUM is able to detect a change to any distribution whose distance from the pre-change distribution is above a user-supplied threshold.
Our main theoretical results (Theorem \ref{prop:kcuprop-abstract} and Corollary \ref{prop:kcuprop})
concern the delays and false alarms of the KCUSUM.  We derive an upper bound on the time to detect a change (that is, the delay) and a lower bound on the time until a false alarm occurs when there is no change. 
The analysis builds on existing theory  for the CUSUM \cite{lorden1970excess,lorden1971}.

The rest of this paper is structured as follows. In Section \ref{sect:cusum} we review the basic notions of the CUSUM algorithm and  in Section \ref{sect:mmd} we review the MMD framework. We introduce the KCUSUM algorithm in Section \ref{sect-two}, where we also present the analysis. In Section \ref{sect-emp} we present the results of a numerical evaluation.
\section{Cumulative Sum algorithm\label{sect:cusum}}
We consider a sequence of random variables
$\{x_n\}_{n\geq 1}$
and assume that there is an index $t$ such that for all $1\leq i <t$ the variables $x_i$ have the distribution $p_0$, and for $i\geq t$, the $x_i$ have the distribution $p_1$. Presently we assume the $x_i$ take values in some Euclidean space
$\mathbb{R}^d$,
although the kernel methods that we shall introduce are not restricted to this scenario. The index $t$ is referred to as the \textit{change point}. An online change detection algorithm tries to identify this change point in real-time, and bases the decision of whether or not a change has occurred by time $n$ on the data available up to time $n$.

\begin{figure*}[]
  \centering
  \hspace{-2em}
\includegraphics[width=0.35\linewidth]{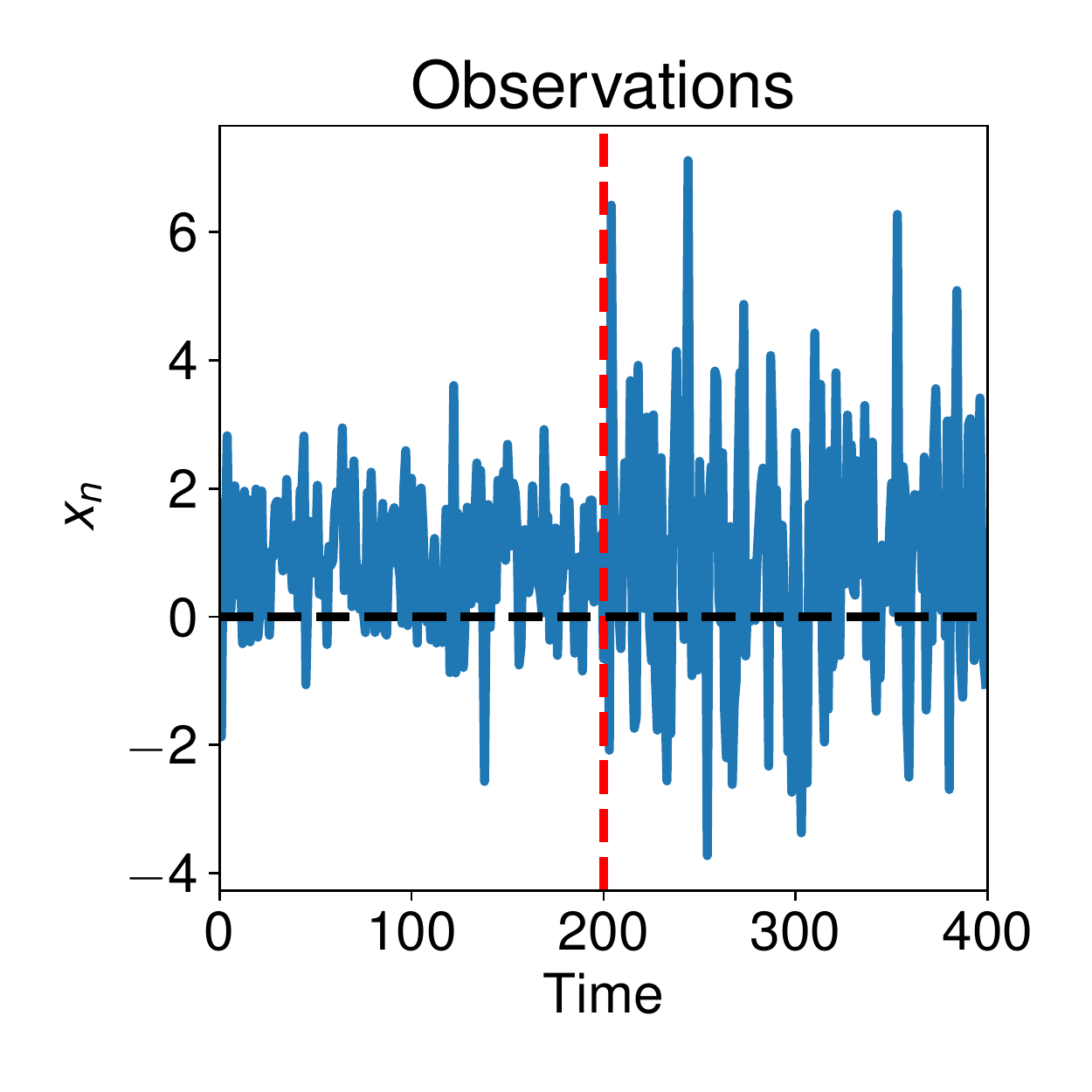}\hspace{-1em}
\includegraphics[width=0.35\linewidth]{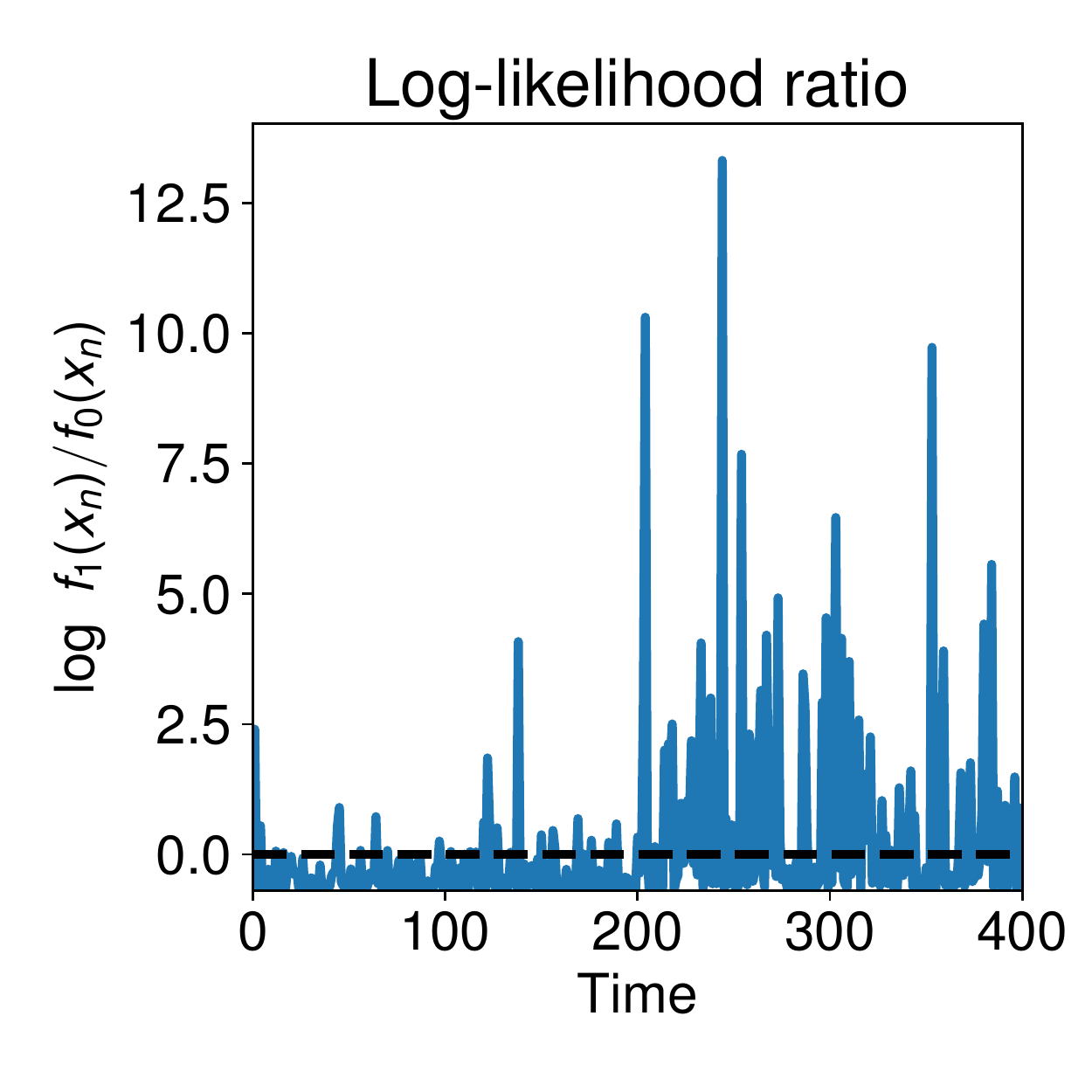}\hspace{-0.5em}
\includegraphics[width=0.35\linewidth, trim=0cm -0.0cm 0 -0cm]{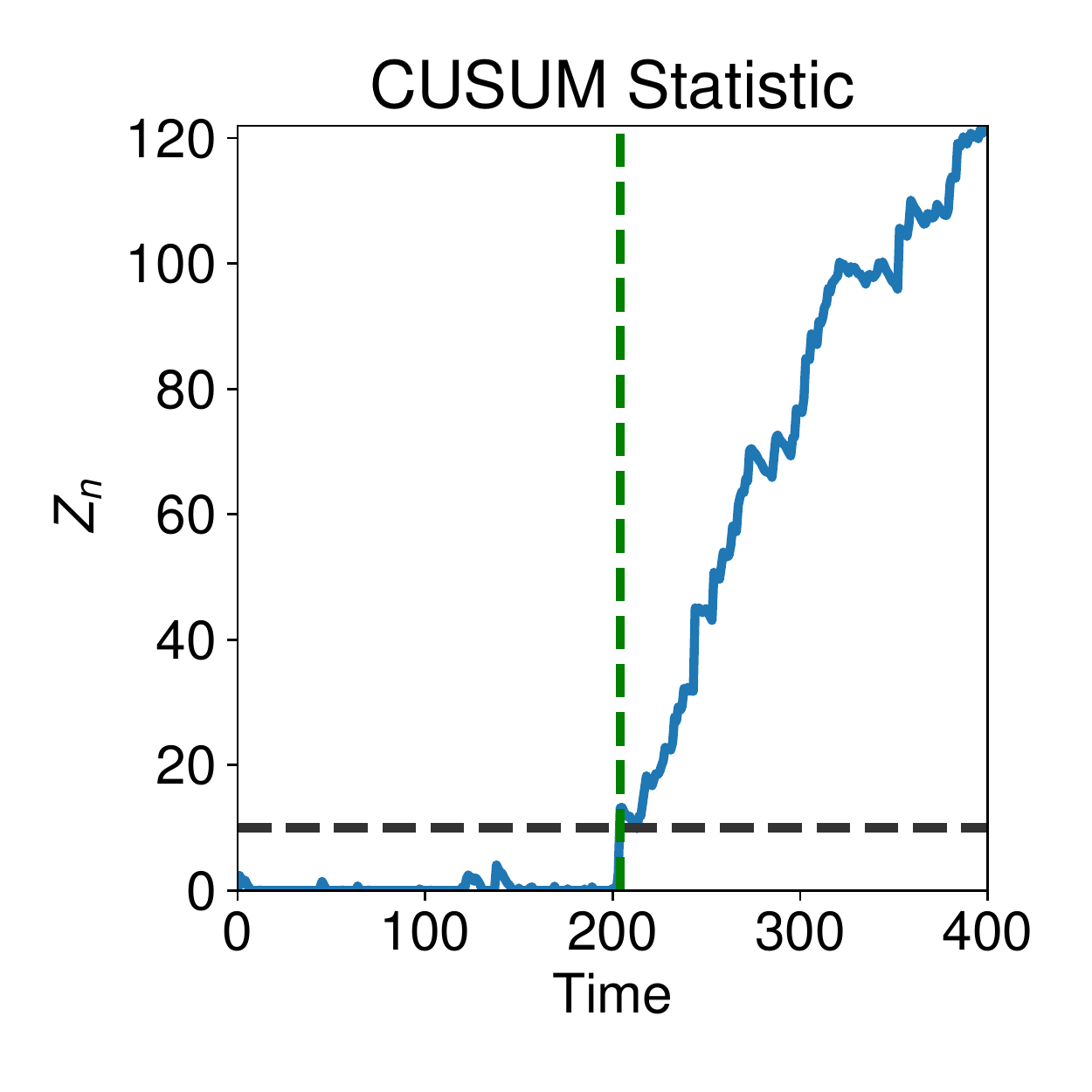}
\caption{Illustration of the CUSUM. For each observation (left) the log-likelihood ratio is computed  (middle) and added to a running sum (right). See main text for details.
  \label{myfig}}
\end{figure*}

The Cumulative Sum algorithm is an online change detection procedure introduced by Page \citep{page1954continuous}. For the purposes of introducing the CUSUM,  assume that the distributions $p_0$ and $p_1$ have densities $f_0$ and $f_1$ respectively.
 The steps of the CUSUM are  presented in Algorithm \ref{algo:cusum}. 
 At step $n$ of the procedure a data point $x_n$ is observed, and the log-likelihood ratio $\log \frac{f_1(x_n)}{f_0(x_n)}$ is calculated and the result added onto the statistic $Z_n$. If the result would be negative then $Z_n$ is set to zero,  effectively restarting the algorithm. If $Z_n$ crosses a threshold $h$ then a change is declared at time $n$. Formally, the CUSUM stopping time is
 $T_{\textrm{CUSUM}} = \inf\{ n\geq 1 \mid Z_n \geq h\}$.

For some insight into why the CUSUM works, consider the behavior of the log-likelihood ratio before and after the change. Before the change,  it
has mean $-d_{KL}(p_0,p_1)$, where
$
d_{\mathrm{KL}}(p_0,p_1)
=
\mathbb{E}_{p_0}\left[\log \frac{f_0(x)}{f_1(x)}\right]
$
is the Kullback-Liebler divergence between the distributions $p_0$ and $p_1$.
Since $d_{KL}$ is positive when $p_0\neq p_1$, the increment term will have  a negative mean.
This drift combined with the barrier at zero causes the statistic to stay near zero before the change. After the change, the increment has a positive mean equal to
$d_{KL}(p_1,p_0)$
and
$Z_n$ begins to increase, eventually crossing any positive threshold $h$ with probability one, which will cause the algorithm to end.
Beyond these heuristic arguments, the CUSUM can be shown to be optimal in a certain sense, as we review below.

\begin{algorithm}[t]
\caption{CUSUM Algorithm \citep{page1954continuous}\label{algo:cusum}}
\DontPrintSemicolon
\textbf{input:} Data $x_1,x_2,\hdots $ and threshold $h\geq 0$.
\\
\textbf{initialize} $Z_0 = 0$
\\
\textbf{for} $n=1,2,\hdots$ \textbf{do} \\
\quad 
$
Z_{n}
=
\max\left\{
0, Z_{n-1} + \log \frac{f_1(x_n)}{f_0(x_n)}
  \right\}$.
  \\
\quad  \textbf{if } $Z_n \geq h$ \textbf{then} set $T_{\textrm{CUSUM}} = n$ and \textbf{exit}.
  \\
\quad  \textbf{else } continue.
\\
\textbf{end}
\end{algorithm}
\begin{example}\label{cu-ex}
Denote by $\mc{N}(a,b)$ the normal distribution with mean $a$ and variance $b$.  Consider detecting a change in variance of normally distributed random variables, where the pre-change distribution is $\mc{N}(1,1)$ and the post-change distribution is $\mc{N}(1,4)$.
 A sample sequence $x_1,\hdots,x_n$ of length $n=400$  with a change point at $t=200$ is shown on the left in Fig. \ref{myfig}. The true change time is marked by a dashed line. The log likelihood ratio in this case is 
$\log\frac{f_1(x)}{f_0(x)} = \frac{3}{8}x^2 - \frac{3}{4}x + \log\frac{1}{2} + \frac{3}{8}$. The values of $\log\frac{f_1(x_n)}{f_0(x_n)}$ for $n=1,2,\hdots$ are plotted in the middle of Fig. 
 \ref{myfig}. We can see that the log likelihood ratio has a negative mean before the change and a positive mean after the change. The resulting CUSUM statistic $Z_n$ is shown on the right of Fig. \ref{myfig}. Using a threshold of $h=10$ results in detection at time $T_{\textrm{CUSUM}}=212$.
\hfill\ensuremath{\blacksquare}
\end{example}

Next we review the performance characteristics of the CUSUM.
Each possible change time defines a different distribution on the sequences 
$\{x_n\}_{n\geq 1}$.
If there is no change, then the variables are independent and identically distributed (i.i.d.) with
$x_i \sim p_0$ for all $i\geq 1$.
We denote this distribution on sequences by
$\mathbb{P}_{\infty}$, and denote expectations with respect to this distribution by
$\mathbb{E}_{\infty}$.
In general, a change at time
$ t \geq 1 $
means that for $1 \leq i \leq t-1$ the $x_i$ are i.i.d. with
$x_i \sim p_0$ and for $i \geq t$ they are i.i.d. with
$x_i \sim p_1$.
We let 
$\mathbb{P}_{t}$ denote the probability distribution on sequences under the assumption of a  change at time $t$, and $\mathbb{E}_{t}$ represents the expectation under this distribution.
For
$n\geq 1$ let
$\mc{F}_n$ be the $\sigma$-algebra
$\mc{F}_n = \sigma(x_1,x_2,\hdots,x_n)$.
Intuitively, $\mc{F}_n$ represents the information contained in the observations up to and including time $n$. Formally, an online change detector can be represented as a stopping time  with respect to the filtration $\{\mc{F}_n\}_{n\geq 1}$, that we denote by $T$, with the interpretation that the value of $T$ is an estimate of the change point.

When running a change detector on a particular sequence, two types of errors may occur. There may be a false alarm, which means the change is detected too early, or there may be a delay, meaning the change was detected late. We formalize the levels of false alarm and delay using the metrics of average run length to false alarm ($\operatorname{ARL2FA}$)  and worst case average detection delay ($\operatorname{ESADD}$). These are standard metrics for evaluating change detectors \citep{lorden1971,olympiabook,moustakidesnumerical}.

For a change detector $T$, the average time to false alarm is 
\begin{equation}\label{ttfa}
  \operatorname{ARL2FA}=\mathbb{E}_{\infty}[T].
\end{equation}
That is,
the $\operatorname{ARL2FA}$
is the average amount of time until a change is detected given a  sequence of observations with no change.

We measure delay using Lorden's criterion \cite{lorden1971}.
If
$T$
is a change detector and there is a change at time
$t\geq 1$,
then the expected delay given the history of the observations up to time
$t-1$
is the random variable
$
\mathbb{E}_t[
  ( T -t)^{+} \mid \mc{F}_{t-1}
].
$
\footnote{The notation $(\cdot)^{+}$ refers to the positive part function: $(x)^+ = \max\{0,x\}$.}
The worst case delay for a change at time
$t$
is obtained by taking an essential supremum over all possible sequences of length
$t-1$,
denoted by
$\esssup \mathbb{E}_t[(T-t)^{+} \mid \mc{F}_{t-1}]$.
Finally, taking the supremum over all change times
$t$
we obtain the worst case delay:
\begin{equation}\label{delay}
  \begin{split}
    &\operatorname{ESADD} =
  \sup_{1\leq t <\infty}\esssup \mathbb{E}_{t}[ (T - t)^{+} \mid \mc{F}_{t-1}].
  \end{split}
\end{equation}
Notably, the CUSUM provides the optimal trade off between the time to false alarm and the worst case delay. This was first proved in an asymptotic form in \citep{lorden1971}, and the result was later proved in full non-asymptotic form in \citep{moustakides1986optimal}. A proof of optimality can be found in
 \citep{moustakides1986optimal,olympiabook}. Further information on the derivation and properties of the  CUSUM may be found in \cite{Basseville1993}, or \cite{olympiabook}.

The precise relation between the threshold $h$ and the performance levels ARL2FA and ESADD is non-trivial and involves solving numerically intractable integral equations \citep{page1954continuous,moustakidesnumerical}. However, it is possible to derive some upper and lower bounds that may be useful in practice.  For the sake of comparison with the analysis of Kernel CUSUM, it will interesting to consider the following quantitative bounds on the performance of the CUSUM. 
\begin{prop}\label{cusum-prop}
  The performance of the $\mathrm{CUSUM}$ (Algorithm \ref{algo:cusum}) can be bounded as follows.
  The time to false alarm satisfies
  $$\operatorname{ARL2FA}_{\mathrm{CUSUM}} \geq \exp(h).$$
  If it also holds that
  $\mathbb{E}_1\left[ \big(\big(\log\frac{f_1(x)}{f_0(x)}\big)^{+}\big)^{2}\right] < \infty$ then
  \begin{equation*}
    \begin{split}
      &\operatorname{ESADD}_{\mathrm{CUSUM}}
      \leq \\&\quad
      \frac{h}{d_{\mathrm{KL}}(p_1,p_0)} +
      \frac{1}{d_{\mathrm{KL}}(p_1,p_0)^{2}}
      \mathbb{E}_1\left[ \left( \left( \log\tfrac{f_1(x)}{f_0(x)} \right)^{+}\right)^2 \right].
    \end{split}
\end{equation*}
\end{prop}\vspace{0.5em}
\begin{proof}
  See the appendix.
  \end{proof}
The intuitive interpretation of these equations is that increasing the threshold $h$ causes an increase the time until false alarm, but it also leads to increased detection delay. From the second equation, we can see that the detection delay increases when the distributions get closer.
The term involving the positive part of the log-likelihood ratio is related to the variance of the CUSUM statistic. In Corollary \ref{prop:kcuprop} below we shall obtain somewhat analogous bounds for the Kernel CUSUM.
\section{Maximum Mean Discrepancy\label{sect:mmd}}
Two-sample testing refers to the problem of determining whether two datasets are drawn from the same distribution. One approach to this problem is to consider the empirical measures defined by the datasets, and to compute the distance between the empirical measures using a probability metric. If enough data points are used, then the empirical distance should be close to the true distance. If the distance is large then we can be confident that the datasets are generated by different distributions. This is the idea underlying several classical tests, such as
the Kolmogorov-Smirnov test \citep{feller1948kolmogorov}, 
the Cramer-von-Mises test \citep{anderson1962}
and the Anderson-Darling test
\citep{adarling}.

The  Maximum Mean Discrepancy (MMD) approach is also based on computing the distance between empirical distributions. In MMD, the datasets are implicitly embedded in a Reproducing Kernel Hilbert Space (RKHS) corresponding to a user-supplied kernel function $k$, and the distance between the embeddings is computed \citep{gretton2012kernel}. Compared to classical approaches, there are several features of MMD that make it appealing for non-parametric statistics.  First, the MMD distance has a range of simple unbiased estimators (see the definition of $\rho_L$ below for one such example.) Second,
there is the flexibility offered by choice of kernel, and using kernels means the test can be applied on datasets without a natural Euclidean representation, such as strings, graphs and other structured data \cite{gartner2003survey, vishwanathan2010graph}.

Let $\mathcal{X}$ be a set, and let
$k :\mathcal{X}\times\mathcal{X} \to \reals$
be a kernel on this set;
this is a symmetric, positive definite function that we regard intuitively as a similarity measure\footnote{Symmetric means that
  $k(x,y) = k(y,x)$
  and positive-definite means that for any choice of
  $n$
  elements
  $x_1,\hdots,x_n$
  and
  $n$
  real numbers
  $a_1,\hdots,a_n$,
  we have
  $\sum\limits_{i=1}^{n}\sum\limits_{j=1}^{n} a_i a_j k(x_i,x_j) \geq 0$}.
The reader may have in mind the set $\mathcal{X} =\reals^n$ and the Gaussian kernel
\begin{equation}\label{eqn:gauss-kern}
  k(x,y)= \exp(-\|x-y\|^{2}/2)
\end{equation}
Some other choices for the kernel function may be found in \citep[Table~3.1]{meanembeddingreview}.
Further suppose that
$\mathcal{X}$
has the structure of a measurable space
$(\mathcal{X},\Sigma)$ and that $k$ is a measurable function on
$\mathcal{X}\times \mathcal{X}$ with the product $\sigma$-algebra.
Define $\mathcal{P}(\mathcal{X})$ to be the set of all probability measures on $(\mathcal{X},\Sigma)$,
and using the kernel $k$,  define the subset 
$
\mathcal{P}_k
=
\{ \mu \in \mathcal{P}(\mathcal{X}) \mid  \mathbb{E}_{x\sim\mu}[\sqrt{k(x,x)}] < \infty \}$.
If the kernel has the additional property of being characteristic\footnote{We refer the reader to \citep{gretton2012kernel} or \citep{meanembeddingreview} for a precise definition of characteristic kernel. For instance, if $\mathcal{X} = \mathbb{R}^n$ then the Gaussian kernel  is characteristic.}
then we may define the MMD metric on $\mathcal{P}_k(\mathcal{X})$, denoted $d_k$. This metric is defined as
 \begin{align*}
     &d_k(p_0,p_1) =  \\
   &
   \sqrt{\mathbb{E}_{p_0\times p_0}[k(x,x')]
     +\mathbb{E}_{p_1\times p_1}[k(y,y')]
     - 2\mathbb{E}_{p_0\times p_1}[k(x,y)]}.
\end{align*}
See \citep{gretton2012kernel} for more details. 

 One of the unbiased estimators of $d_k^2$ presented in \citep{gretton2012kernel} is the linear statistic
 $\rho_{\textrm{L}}$ defined below.
For convenience, the estimator is expressed using the following function $g$:
\begin{equation}\label{def:h}
  \begin{split}
g( (x_0, x_1), ( y_0, y_1) ) =
&k(x_0,x_1) + k(y_0,y_1)\\
&- k(x_0,y_1) - k(x_1,y_0).
\end{split}
\end{equation}
Consider  two data sets
$X = \{x_1,x_2,\hdots,x_n\} \setin \mc{X}$ and
$Y = \{ y_1,y_2,\hdots,y_n\} \setin \mc{X}$.
Then $\rho_{\textrm{L}}(X,Y)$ is
 \begin{equation}\label{rnd-walk}
  \begin{split}
    \rho_{\textrm{L}}(X,Y) &=
    \frac{1}{\lfloor n/2 \rfloor}\sum_{i=1}^{\lfloor n/2 \rfloor}
    g( (x_{2i-1},x_{2i}),(y_{2i-1},y_{2i})).
  \end{split}
\end{equation}
The linear statistic is interesting as it is a sum of i.i.d. terms, which means the central limit theorem may be used to approximate its distributional properties, which can help in tuning the thresholds when MMD is used for hypothesis testing \citep{gretton2012kernel}.
Furthermore, in the online setting we can study the trajectory of the statistic (that is, as a function of $n$) using the theory of sums of i.i.d. random variables, greatly facilitating analysis. We will exploit this structure in the analysis of KCUSUM below.

\subsection{Other related work \label{sect:other-work}}
The field of online change detection has its roots in industrial quality control \citep{shewhart1931economic}, and sequential analysis \cite{wald1947sequential}. In particular, the CUSUM procedure is closely related to the sequential probability ratio test (SPRT), which is a foundational online hypothesis testing algorithm \cite{wald1947sequential}.

An alternative to online change detection is offline detection, where the  algorithms do not run until the entire sequence is observed, and all data is used to make a decision. Kernel offline change  point analysis was explored in \citep{harchaoui2009kernel}.
Each hypothetical change time  is used to partition the dataset into two groups, consisting of prior observations and later observations, and the two resulting datasets are compared using a kernel based test. This is repeated for each possible change time. If one of the comparisons yields a significant discrepancy between the pre- and post- observation datasets, then a change is declared at the point where the difference was largest.
This procedure is of interest as it does not use reference data, instead basing its decisions on comparisons between disjoint sets of observations.

A number of approaches to non-parametric online change detection have been proposed in \citep{brodsky1993nonparametric}. As given, they apply only to the case of detecting changes in mean. One straightforward method of applying kernel non-parametric tests to online change detection is a sliding window approach, as pursued in \citep{mstatistic}. In this approach, at each time a decision regarding the change is made based on the distance between the most recent fixed-size block of data and a block from the reference distribution, using an MMD distance $d_k$. This can be seen as an kernel-based generalization of Shewhart control charts \citep{shewhart1931economic}. 

Besides the CUSUM, there are other algorithms with optimality properties, notably the Shiryaev-Roberts (SR) change detector \citep{shiropt}. While the CUSUM minimizes the worst case delay (ESADD), the SR detector minimizes a form of average delay. Like the CUSUM, the  SR test statistic admits a simple recursive form and hence it may be possible to extend our algorithmic construction to SR-type detectors as well, but this is outside the scope of this paper.

 \section{The Kernel CUSUM Algorithm}\label{sect-two}
 \begin{figure*}[]
   \centering
   \hspace{-2em}
\includegraphics[width=0.35\linewidth]{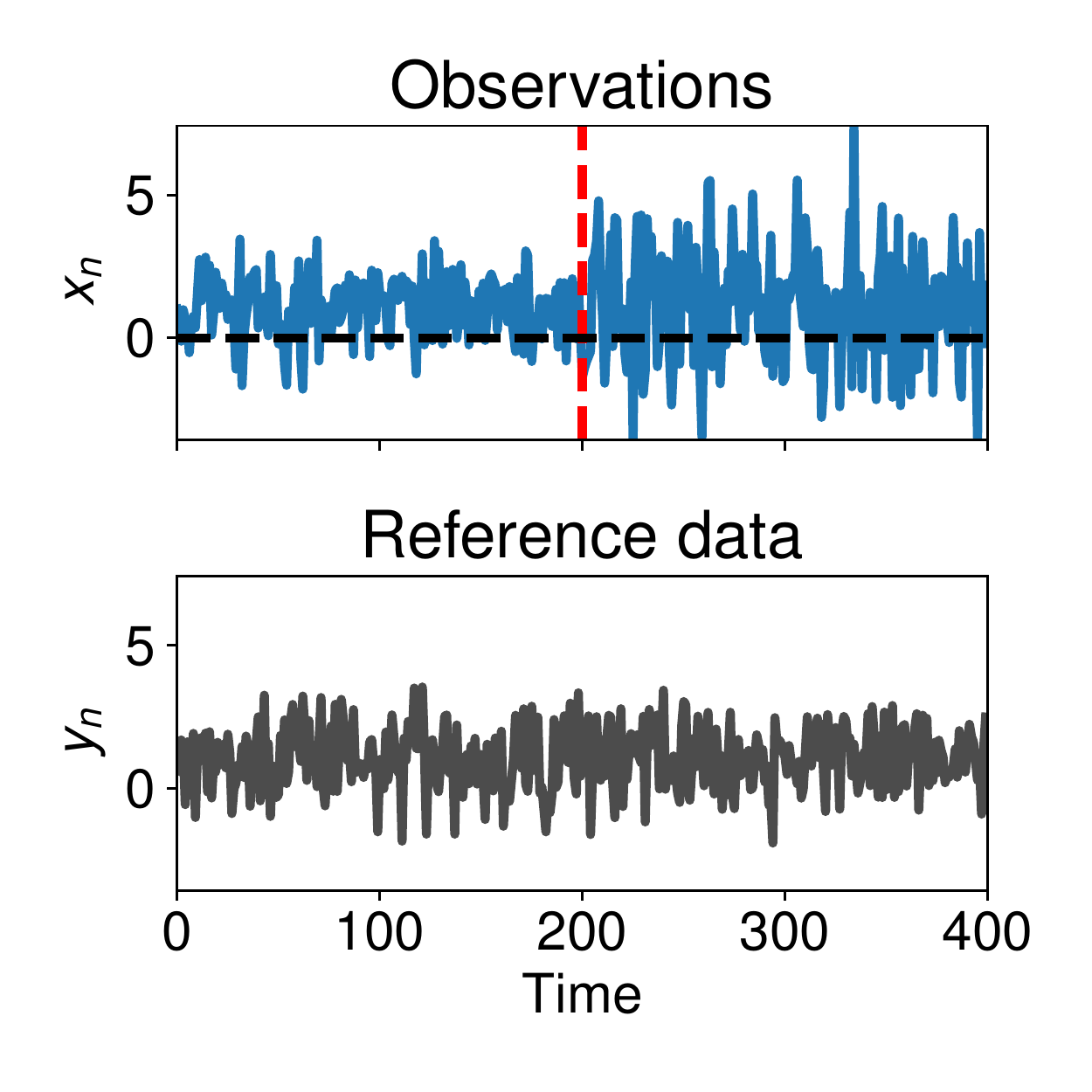}\hspace{-1em}
\includegraphics[width=0.35\linewidth]{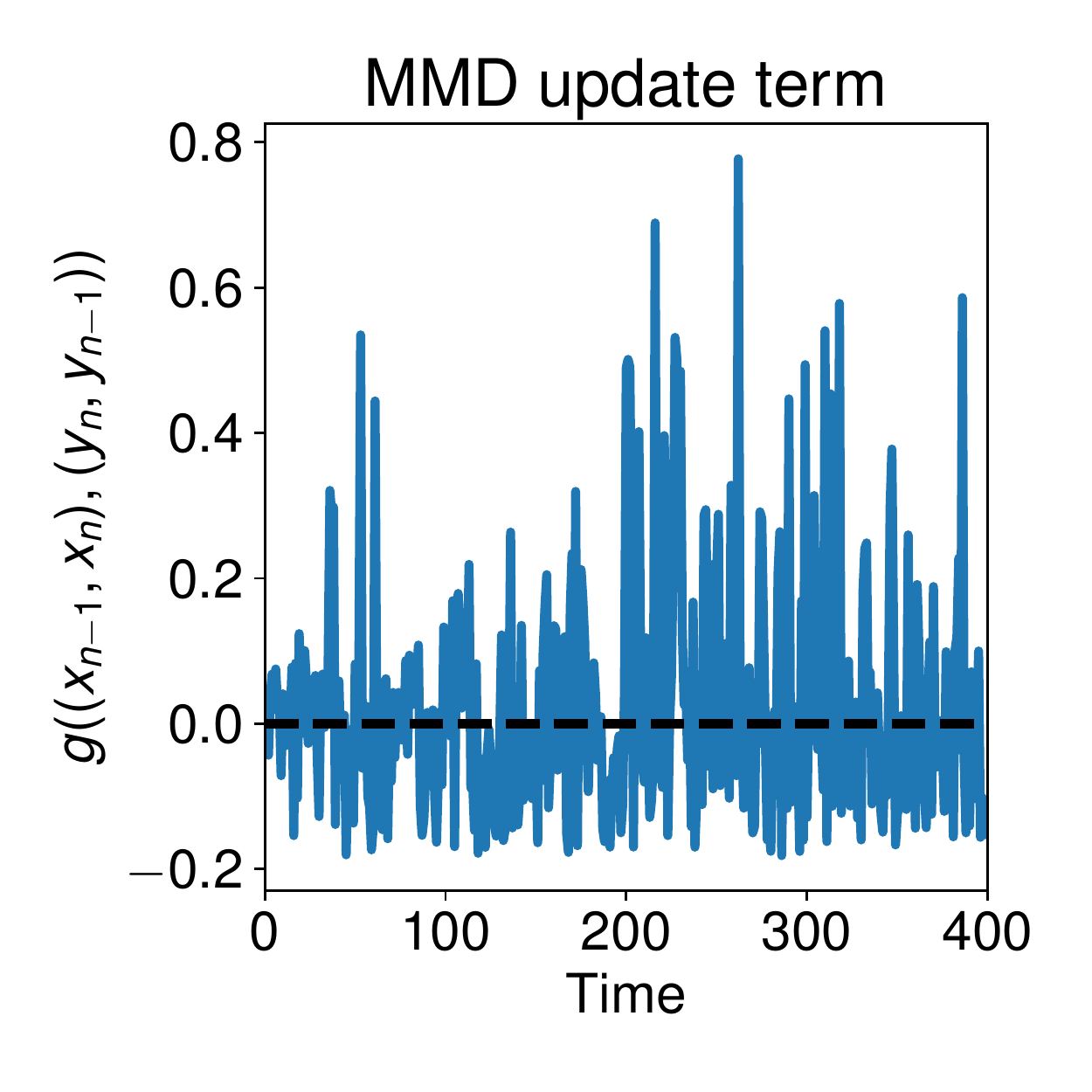}\hspace{-0.5em}
\hspace{0.5em}\includegraphics[width=0.33\linewidth,trim=0cm -0.8cm 0 -0cm]{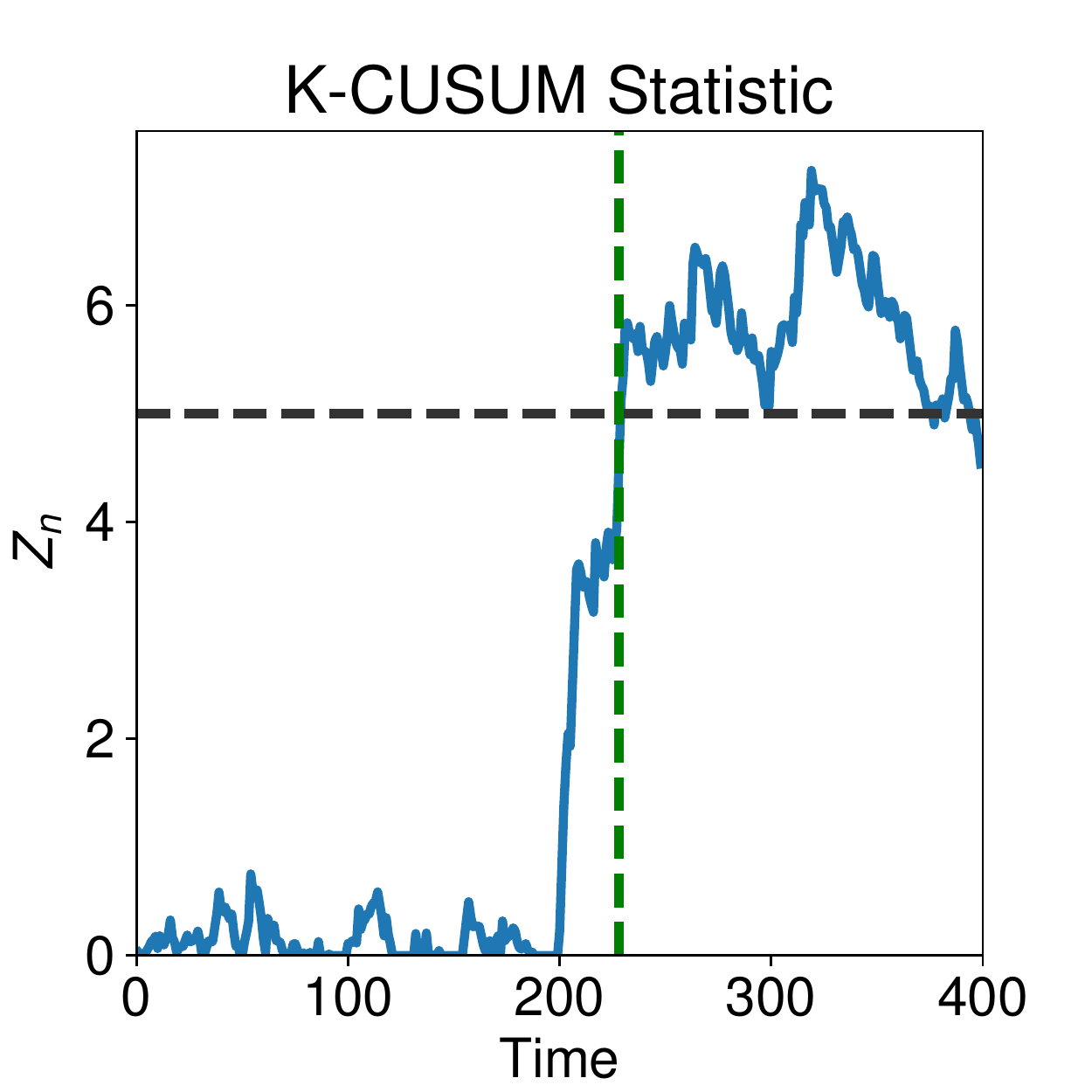}
\caption{ Illustration of the Kernel CUSUM. The observations (left upper) are compared with reference data (left lower) by computing the MMD statistic (middle).  A change is detected when the cumulative sum of the MMD comparisons crosses a threshold. See the main text for details.  \label{kfig}}
\end{figure*}
 The Kernel CUSUM (KCUSUM) algorithm blends features of the CUSUM procedure  with the MMD framework.
 The basic idea is that instead of using each new observation $x_n$ to compute the log-likelihood ratio
 $\log \tfrac{ f_1(x_n)}{f_0(x_n)}$, which is an estimator of the KL-divergence
 $d_{KL}$, we will use the new observation and some other random samples to compute an estimate of an MMD distance
 $d_k$.
 
 The KCUSUM algorithm is defined with the help of a shifted version of the function $g$ used to define the MMD statistics:
 Given
 $\delta>0$
 (the role of $\delta$ is explained in detail below), define $g_{\delta}$ as
 \begin{equation}\label{def:shiftg}
  \begin{split}
g_{\delta}( (x_0, x_1), ( y_0, y_1) ) =
&\,k(x_0,x_1) + k(y_0,y_1)\\
&- k(x_0,y_1) - k(x_1,y_0) -\delta.
\end{split}
 \end{equation}
 The details of the Kernel CUSUM are listed in Algorithm \ref{algo:twosampcd}.
 It is assumed that the change to be detected is from a reference distribution $p_0$ to an unknown distribution $p_1$.
 At even numbered iterations,
the most recent observation $x_n$ is paired with the previous observation
$x_{n-1}$
and these two points are compared with two reference points
$y_n,y_{n-1}$
using MMD. We subtract a constant
$\delta>0$
from the result to get $v_n$. The variable $v_n$ is  then added onto the statistic
$Z_{n-1}$
to get the next value
$Z_n$.
If the new value $Z_n$ would be negative, then it resets to $0$, effectively restarting the algorithm.
At odd numbered iterations, the statistic $Z_n$ is unchanged. Note that as a consequence, the algorithm only raises alarms at even numbered iterations.

The reason for subtracting a positive amount
$\delta$
at each step of
Algorithm \ref{algo:twosampcd}
is to guarantee that the increments $v_n$  have negative drift under the pre-change regime and positive drift in the post-change regime. This is a consistency property that enables us to formulate non-trivial bounds on performance.
Using the definitions in Algorithm \ref{algo:twosampcd}, it is evident that before a change,
$\mathbb{E}[v_n] = -\delta < 0$
and after a change,
$\mathbb{E}[v_n] = d_{k}^{2}(p_0,p_1) - \delta$.
In other words,  the algorithm can detect a change to any distribution
$p_1$
that is at least distance
$\sqrt{\delta}$
away from the reference distribution $p_0$.

\begin{algorithm}[t]
  \caption{Kernel CUSUM (KCUSUM) \label{algo:twosampcd}}
  \DontPrintSemicolon
  \textbf{input:}
  Thresholds $h\geq 0,\delta>0$
  and data
  $ x_1,x_2,\hdots$
  \\
  \textbf{initialize} $Z_1 = 0$
  \\
  \textbf{for} {$n=2,3,\hdots$} \textbf{do} \\
   \quad \textbf{sample} $y_n$ from reference measure $p_0$
    \\
    \quad \textbf{ if} $n$ is even \textbf{then } \\
    \quad   \quad
    $v_n = g_\delta(  (x_{n-1}, x_n), (y_{n-1},y_n) ) $. \\
    \quad \textbf{else} \\
    \quad\quad $v_n = 0 $ \\
     \quad \textbf{end}
    \\    
    \quad $Z_n = \max\{ 0, Z_{n-1} +  v_n\}$ \\
    \quad \textbf{if} $Z_n > h$ \textbf{then} set $T_{\textrm{KCUSUM}} = n$ and \textbf{exit}
    \\
    \quad \textbf{else} continue.
    \\
  \textbf{end}
\end{algorithm}

\begin{example}\label{ex:depict}
 Consider the problem of detecting a change in variance of normally distributed random variables, as in  in Example \ref{cu-ex}.
  The upper left plot of  Fig. \ref{kfig} shows the stream of observations, which are normally distributed with a change in variance at time $t=200$.
  The KCUSUM is based on the linear statistic
  $\rho_L$,
  so at each time $n$ the MMD estimate
  $k(x_{n-1},x_{n}) + k(y_{n-1},y_{n}) - k(x_n,y_{n-1}) - k(x_{n-1},y_n)$ is computed, and this quantity is plotted
   in the middle of Fig. \ref{kfig}.
  On the right is the resulting $Z_n$ sequence.
  As in the CUSUM, a simple threshold is used to decide that a change has occurred.
  For this particular realization of the variables, a threshold of $h=5$ results in detection at
  $T_{\textrm{KCUSUM}} = 225$.
  The value of $\delta$ was $\delta=1/40$.
\hfill\ensuremath{\blacksquare}
\end{example}

\,

Next we consider  the delay and false alarm rate of the KCUSUM.
The time to false alarm is defined as in Equation \eqref{ttfa}.
The worst case delay is defined as in Equation \eqref{delay}, using the filtration
$\{\mc{F}_{n}\}$ where
$\mc{F}_n = \sigma(x_1,y_1,\hdots,x_n,y_n)$.
To prove the bounds  we adapt the methods of \citep{lorden1971}, which
allows one to reduce the problem of analyzing the CUSUM to that of  analyzing a random walk with i.i.d. terms.

For the analysis it will be convenient to group the variables together as $\{z_n\}_{n\geq 1}$ where for $n\geq 1$,
  \begin{equation}\label{zdef}
    z_n = ( x_{2n-1},x_{2n},y_{2n-1},y_{2n}).
    \end{equation}
  The grouping reflects how the algorithm process the data in blocks of two. As a consequence, the bounds in this section involve additional factors of two compared the CUSUM (see Proposition \ref{cusum-prop}).

  Associated to the grouped sequence $\{z_n\}_{n\geq 1}$, define auxiliary stopping times $c_1,c_2,\hdots$ as
 $$c_n = \inf \left\{ k \geq n\,  \middle|\,  \sum\limits_{i=n}^{k}g_{\delta}(z_i) > h\right\}.$$ 
\begin{thm}\label{prop:kcuprop-abstract}
  Let
  $T_{\textrm{KCUSUM}}$
  be the change detector corresponding to the Kernel CUSUM 
  (Algorithm \ref{algo:twosampcd}).
  Let
  $p_0$ be the pre-change, or reference distribution.
  Then by using a threshold
  $\delta > 0$,
  the time to false alarm is at least
  \begin{equation}\label{kcuttfa-abstract}
    \operatorname{ARL2FA}_{\mathrm{KCUSUM}} \geq \frac{2}{\mathbb{P}_{\infty}(c_1 < \infty)}.
  \end{equation}

  If $p_1$ is a distribution with
  $d_k(p_0,p_1) > \sqrt{\delta}$, then
  the worst case detection delay is at most
\begin{equation}\label{kcuwcd-abstract}
  \operatorname{ESADD}_{\mathrm{KCUSUM}} \leq 2\mathbb{E}_{1}[c_1].
\end{equation}
\end{thm}
\,
\begin{proof}
    Let
  $b_0=0$ and for $n\geq 1$ let
  $b_n = \max\{ 0 , b_{n-1} + g_{\delta}(z_n)\}.$
 Define the stopping time $c$ as 
 \begin{equation}\label{eqn:c}
   c = \inf\{ n\geq 1 \mid b_n > h\}.
 \end{equation}
 The relation between $c$ and the  KCUSUM stopping time  $T_{\textrm{KCUSUM}}$ is
    \begin{equation}\label{reduced-relation}
      T_{\textrm{KCUSUM}} = 2c.
      \end{equation}
 Note that, as discussed in \cite{lorden1971}, the stopping time $c$ can be represented as
        \begin{equation}\label{c-repr}
          c  = \inf_{n\geq 1} c_n.
        \end{equation}
       and each stopping time
    $c_n$
     uses the same decision rule, the only difference being that they operate on shifted versions of the input sequence
    $\{x_i\}_{i\geq n}$.
     In this setting Theorem 2 from \cite{lorden1971} is applicable, which yields a lower bound on
     $c$
     under
     $\mathbb{P}_{\infty}$:
 \begin{equation}\label{lorden-res}
   \mathbb{E}_{\infty}[c] \geq \frac{1}{\mathbb{P}_{\infty}(c_{1} < \infty)}.
   \end{equation}
 Combining
 Equations \eqref{reduced-relation}
 and
 \eqref{lorden-res} yields the claim \eqref{kcuttfa-abstract}.
 
Now we consider bounding the worst case delay.
If the sequence $\{x_i\}_{i\geq 1}$ has a change at an odd valued time, say $t=2m-1$ for $m\geq 1$, then
the  sequence $\{z_n\}_{n\geq 1}$ has a change at time $m$. Explicitly,
  $$z_{i} \sim \begin{cases}
     p_0 \times p_0 \times p_0 \times p_0 &\text{ for } 1\leq i < m, \\
     p_1 \times p_1 \times p_0 \times p_0  &\text{ for } m \leq i.
     \end{cases}$$
 From here we  reason as in Theorem 2 of \cite{lorden1971}:
 \begin{equation}\label{odd-case}
   \begin{split}
     &\mathbb{E}_{2m-1}[
       (T - (2m - 1))^{+}
       \mid
       \mc{F}_{2(m-1)}] \\
     &\quad\quad\stackrel{\textbf{A}}{=}
     \mathbb{E}_{2m-1}[ 
       (2c- 2m + 1)^{+}
       \mid
       \mc{F}_{2(m-1)}]
      \\
     &\quad\quad\stackrel{\textbf{B}}{\leq}
     \mathbb{E}_{2m-1}[
       (2c_m -2 m + 1)^{+} \mid 
       \mc{F}_{2(m-1)}]
\\
     &\quad\quad\leq
     2\, \mathbb{E}_{2m-1}[
       (c_m - m + 1)^{+} \mid \mc{F}_{2(m-1)}]\\
     &\quad\quad\stackrel{\textbf{C}}{=}
     2\, \mathbb{E}_{2m-1}[
       (c_m - m + 1)^{+} ] \\
    &\quad\quad\stackrel{\textbf{D}}{=}   2 \, \mathbb{E}_{1}[(c_1 - 1 + 1)^{+}]\\
    &\quad\quad=   2 \, \mathbb{E}_{1}[c_1].
   \end{split}
   \end{equation}
 Step \textbf{A} follows from Equation \eqref{reduced-relation} and  Step \textbf{B} follows since $c$ is the infimum of the $\{c_n\}_{n\geq 1}$. Step \textbf{C} follows  from the independence of $c_m$ from    $\mc{F}_{2(m-1)}$ and finally Step \textbf{D} follows from the fact that the distribution of $c_m-m$ under $t  = 2m-1$ is the same as the distribution of $c_1-1$ under $t=1$.
 
 The situation is slightly more complex if the change occurs a time $t$ that is even, say
 $t=2m$
 for some $m\geq 1$. In this case the grouped sequence $\{z_n\}_{n\geq 1}$  does not experience an abrupt change, and instead there are three relevant distributions. Specifically, 
    $$z_i \sim \begin{cases}
    p_0 \times p_0 \times p_0 \times p_0 &\text{ for } 1 \leq i < m, \\
    p_0 \times p_1 \times p_0 \times p_0 &\text{ for } i = m, \\
    p_1 \times p_1 \times p_0 \times p_0 &\text{ for } m < i.
    \end{cases}
    $$
Reasoning as in Equation \eqref{odd-case},
 then,
    \begin{equation}\label{even-case}
      \begin{split}
    &\mathbb{E}_{2m}[(T- 2m)^{+} \mid \mc{F}_{2m -1}] \\
    &\quad\quad= 
        \mathbb{E}_{2m}[(2c- 2m)^{+} \mid \mc{F}_{2m -1}]
        \\
        &\quad\quad\leq
        2 \mathbb{E}_{2m}[(c_{m+1}- m)^{+} \mid \mc{F}_{2m -1}]
        \\
        &\quad\quad=
        2 \mathbb{E}_{2m}[(c_{m+1}- m)^{+} ]
        \\
        &\quad\quad=2 \mathbb{E}_{2}[(c_{2}- 1)^{+} ]  \\
        &\quad\quad=2 \mathbb{E}_{1}[( 1 + c_{1}- 1)^{+} ] \\
        &\quad\quad=2 \mathbb{E}_{1}[c_1 ] .
      \end{split}
    \end{equation}
    Combining \eqref{odd-case} and \eqref{even-case}, we see that for all $t\geq 1$,
    \begin{equation}\label{odd-even-combine}
      \mathbb{E}_{t}[(T - t)^{+} \mid \mc{F}_{t-1}] \leq 2 \mathbb{E}_1[c_1]
    \end{equation}
    Combining Equation \eqref{odd-even-combine}
    with the definition of
    worst case delay \eqref{delay} yields the inequality \eqref{kcuwcd-abstract}.
\end{proof}

We can combine Theorem \ref{prop:kcuprop-abstract} with certain facts about random walks (Lemma \ref{funlemma} and Proposition \ref{lordencor}) to get more specific bounds on the delays and false alarms, as shown in the following Corollary. In this corollary, we assume that the kernel $k$ is bounded by a constant
$\|k\|_{\infty}$, and also assume that
$\delta$ is bounded by
$\delta < 2\|k\|_{\infty}$. This is a necessary assumption when the kernel is bounded, since if
$\delta \geq 2\|k\|_{\infty}$ then
$g_{\delta}(z) \leq 0$
for all $z$, and it will not be possible to detect any changes.
 \begin{figure}[]
   \centering   
\includegraphics[width=1\linewidth]{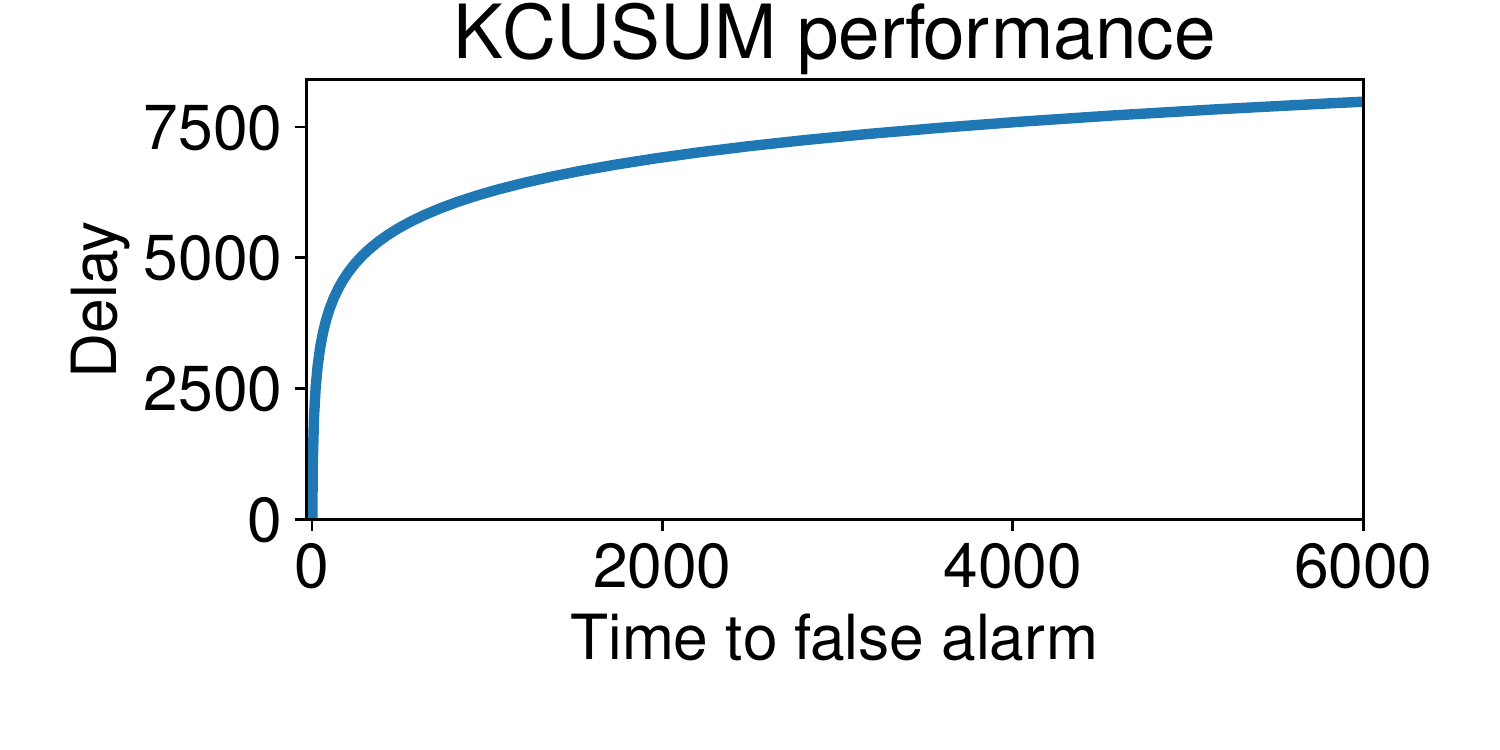}
\caption{ The logarithmic relation between time to false alarm and average delay in the kernel CUSUM as implied by our analysis. See text for details. \label{kcuvis}}
\end{figure}

\begin{cor}\label{prop:kcuprop}
  Let the assumptions of Theorem \ref{prop:kcuprop-abstract} hold.
  Further  assume that the kernel $k$ is bounded by a constant $\|k\|_{\infty}$ and 
  let $\delta < 2\|k\|_{\infty}$.
  Then the time to false alarm satisfies
  \begin{equation}\label{kcuttfa}
    \operatorname{ARL2FA}_{\mathrm{KCUSUM}}
    \geq
    2
    \exp\left(
    \frac{h}{4 \|k\|_{\infty}}
    \log \left(1+\frac{\delta}{4\|k\|_{\infty}}\right)
      \right).
  \end{equation}

  If $p_1$ is a distribution with
  $d_k(p_0,p_1) > \sqrt{\delta}$, then
  the worst case detection delay is at most
\begin{equation}\label{kcuwcd}
  \operatorname{ESADD}_{\mathrm{KCUSUM}}
  \leq
  \frac{2 h}
       {d_{k}(p_0,p_1)^2-\delta}
  +
  \frac{8\|k\|^2_{\infty}}
       {(d_k(p_0,p_1)^2-\delta)^{2}}.
 \end{equation}
  \end{cor}
\begin{proof}
  To start, note that
  \begin{align}\label{inf-repr}
    \mathbb{P}_{\infty}(c_1 < \infty)
    &=
    \mathbb{P}_{\infty}
    \left(
    \inf\left\{
    k \geq 1\bigg| \sum\limits_{i=1}^{k}g_{\delta}(z_i) > h
    \right\} < \infty
    \right) \nonumber \\
    &=
    \mathbb{P}_{\infty}\left(
    \sup_{k\geq 1}\sum\limits_{i=1}^{k}g_{\delta}(z_i) > h
    \right).
  \end{align}

  To upper-bound the last term in this equation,  we will apply Lemma \ref{funlemma}.
  Note that  under our assumption that the kernel is bounded, the moment generating function
  $M(r) = \mathbb{E}_{\infty}[ \exp( r g_{\delta}(z_1) )]$
  is guaranteed to be well defined for all $r\geq 0 $. Therefore
  \begin{equation}\label{what-we-see}
    \mathbb{P}_{\infty}\left(
    \sup_{k\geq 1}\sum\limits_{i=1}^{k}g_{\delta}(z_i) > h
    \right) \leq \exp(-rh),
  \end{equation}
  where $r$ is any number satisfying $r>0$ and $M(r) \leq 1$.
  To identify such an $r$, start with a second order expansion of $M$:
  \begin{align*}
    M(r)
    =
    1
    -
    r\delta
    +
    \int_{0}^{r}
    \int_{0}^{\lambda}
    \mathbb{E}_{\infty}[\exp(u g_{\delta}(z))g_{\delta}(z)^2]\,\mathrm{d} u \, \mathrm{d}\lambda.
  \end{align*}
  Under the assumption that
  $\delta < 2\|k\|_{\infty}$ it holds that
  $|g_{\delta}(z)| \leq 4\|k\|_{\infty}$ and
  \begin{align*}
    &M(r)
    \leq
    1 - r\delta + 16\|k\|_{\infty}^2
    \int_{0}^{r}\int_{0}^{\lambda}\exp(u 4\|k\|_{\infty}) \,\mathrm{d}u \,\mathrm{d}\lambda \\
    &=
    1
    - r\delta
    + 16\|k\|_{\infty}^2
    \int_{0}^{r}\frac{1}{4\|k\|_{\infty}}\left(\exp(\lambda  4\|k\|_{\infty}) - 1\right)
    \,\mathrm{d}\lambda \\
        &=
    1
    - r\delta
    + 4\|k\|_{\infty}
    \int_{0}^{r}
    \left(\exp(\lambda  4\|k\|_{\infty}) - 1\right)
    \,\mathrm{d}\lambda.
    \end{align*}
  Minimizing the right hand side of the final equation above with respect to $r$ yields
  $$
  r
  =
  \frac{1}{4\|k\|_{\infty}}
  \log\left(
  1+ \frac{\delta}{4\|k\|_{\infty}}
    \right).
  $$
Combining \eqref{kcuttfa-abstract}, \eqref{inf-repr}  with \eqref{what-we-see} and using this definition of $r$ yields the claim \eqref{kcuttfa}.

For the delay, note that
$\mathbb{E}_1[c_1]$ is the expected amount of time until  a random walk with positive drift crosses an upper boundary.
Hence we may apply Proposition \ref{lordencor}. This leads to
    \begin{equation}\label{leader}
    \begin{split}
      \mathbb{E}_1[c_1] 
      &\leq
      \frac{h}{\mathbb{E}_1[g_{\delta}(z_1)]}
      +
      \frac{\mathbb{E}_1[(g_{\delta}(z_1)^{+})^2]}{\mathbb{E}_1[g_{\delta}(z_1)]^{2}} \\
      &=
      \frac{h}{d_{k}(p_0,p_1)^2-\delta}+
      \frac{\mathbb{E}_1[(g_{\delta}(z_1)^{+})^2]}{(d_k(p_0,p_1)^2-\delta)^{2}}.
    \end{split}
    \end{equation}

    Combining \eqref{kcuwcd-abstract},  \eqref{leader}, and the bound  
    $g_{\delta}(z)^{+} \leq 2\|k\|_{\infty}$, 
we obtain the relation  \eqref{kcuwcd}.

  \end{proof}
 Figure \ref{kcuvis} shows the logarithmic relation between false alarm time and delay specified by the theorem. For each level of false alarm $x \in \{1,2,\hdots, 10^4\}$, we computed the smallest value of $h$ guaranteed to achieve false alarm rate $x$ according to Equation \eqref{kcuttfa}, and plug in this threshold to compute delay according to Equation \eqref{kcuwcd}.
 The computations were performed for a hypothetical problem where
 $d_k(p_0,p_1)^2 = 1/6, \delta = 2^{-5}$ and $\|k\|_{\infty}=0.5$.
 
 \section{Empirical Results}\label{sect-emp}
  In this section we evaluate the KCUSUM on several change detection tasks.
  In each case, the observations consisted of vectors in $\reals^4$, and the Gaussian kernel (defined in Equation \ref{eqn:gauss-kern}) was used  with $\sigma^2=1$.
The pre-change distribution in each task is the normal distribution on $\mathbb{R}^4$ with mean zero and a covariance matrix equal to the identity scaled by a factor of
$\frac{1}{2}$.  The four possible post-change distributions were as follows:
\begin{enumerate}
\item Change in mean:  A normal distribution with mean $(1,1,1,1)$ and the pre-change covariance matrix.
\item Change in variance (all components): A normal distribution with mean $0$ and a covariance matrix equal to the identity scaled by a factor of $2$.
\item Change in variance (random component): The distribution obtained by sampling from the pre-change distribution and scaling a random component by a factor of $2$.
  \item Change to uniform:  The distribution on $\reals^4$ where each component is sampled independently from the uniform distribution on
$[
  -1/(2\sqrt{3}),
  1/(2\sqrt{3})
]$
\end{enumerate}
  \begin{figure}[]
   \centering
   
\includegraphics[width=1\linewidth]{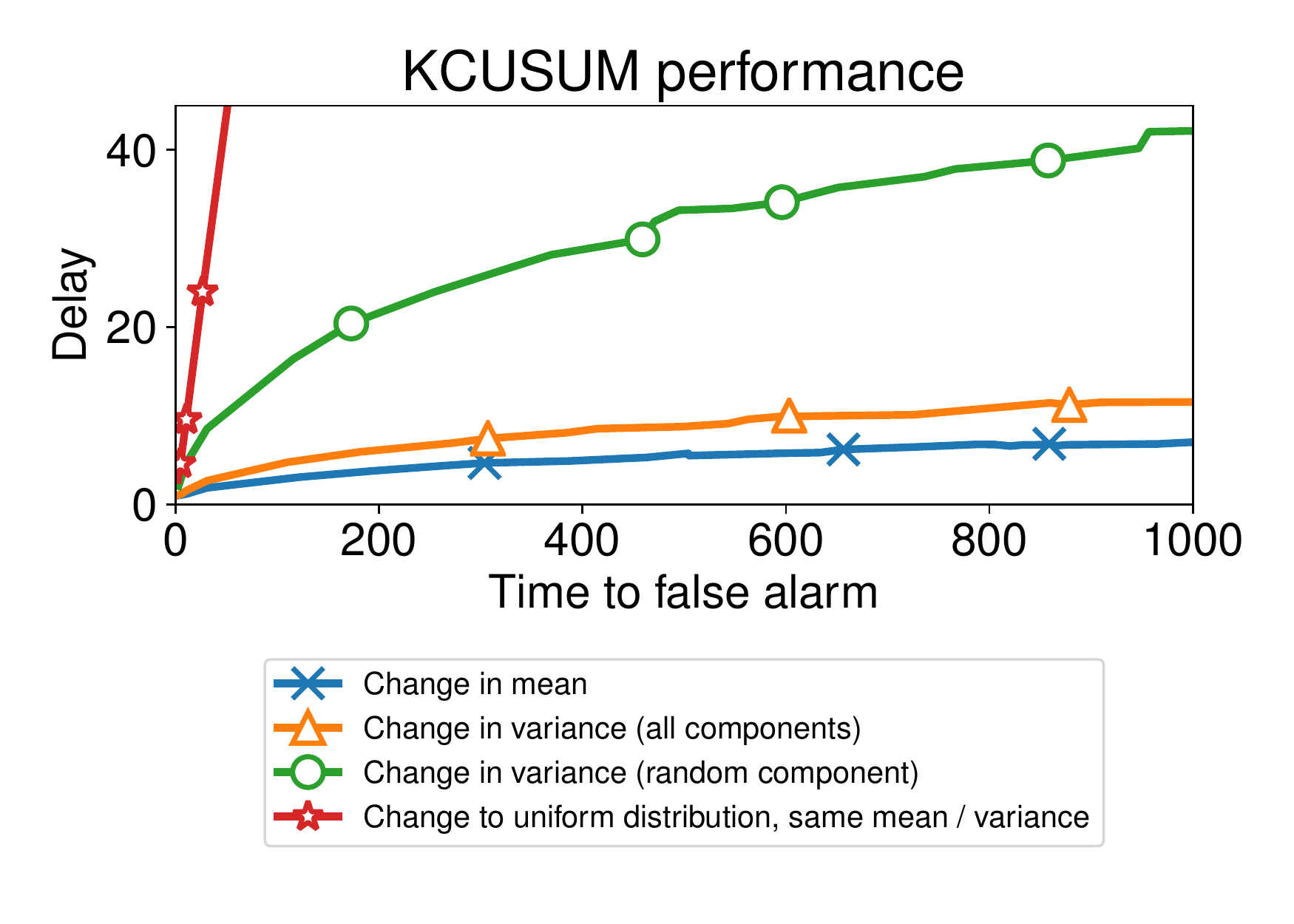}
\caption{ The performance of the KCUSUM on several change detection tasks. See text for details.\label{kcuemp}}
\end{figure}
  Note that the interval in Problem 4 was chosen so that the resulting distribution has the same mean and variance as the pre-change distribution.
  
For each task we used a Monte Carlo approach to estimate the time to false alarm and delay. The time to false alarm was estimated by generating $n=5000$ sequences with no change and running the KCUSUM until a false alarm was detected. We record the time where the false alarm occurs and average theses values to get the estimate for ARL2FA.
For the delay, we  generated $n=5000$ sequences that had a change at time $t=1$, and recorded the amount of time until the alarm goes off as an estimate of the delay. 
In the examples, we set $\delta = 2^{-7}$ in tasks (1) - (3),  and set it $\delta=2^{-9}$ in task (4).

Based on the results in Figure \ref{kcuemp} we see that the change detection tasks increase in difficulty as we go from problem 1 to 4. In all of the problems we observed a similar logarithmic growth rate in the the delay as the time to false alarm is allowed to increase. However, the scale of this growth can vary dramatically. For instance, the first two problems are relatively easy, while the fourth problem seems to be quite difficult for the KCUSUM. Note that in each case, the only data used by the algorithm is the incoming observations and samples from the pre-change distribution, and no information or samples about the post-change distribution are used. Overall, the results suggest that the KCUSUM may be a promising approach for change detection problems where less is known about the type of change.
\section{Conclusion}
This work introduced the Kernel Cumulative Sum Algorithm (KCUSUM), a new approach for online change detection. Unlike the CUSUM algorithm,  this approach does not require knowledge of the probability density ratio for its implementation. Instead, it uses incoming observations and  samples from the pre-change distribution. The result is that the same algorithm works for detecting many types of changes.
 Our theoretical analysis  establishes  the algorithm's ability to detect changes, and shows a relation between the delay and the MMD distance of the two distributions. These bounds should also be useful in the analysis of other non-parametric change detectors. Finally, we would like to suggest two avenues for future work. First, there are likely variants of KCUSUM that leverage more complex, non i.i.d. statistics, that may lead to improved detection performance. Secondly, the CUSUM has been investigated for detecting changes in scenarios with more complex dependencies among observations \cite{fuh2003sprt} and it may also be possible to  extend the kernel methods developed in this paper to detect changes in these cases.

\appendix
\section{Appendix}
 \setcounter{thm}{0}
 \renewcommand{\thethm}{A.\arabic{thm}}
 \begin{prop}[Corollary 1, \citep{lorden1970excess}]\label{lordencor}
  Let $\{a_i\}_{i\geq 1}$ be i.i.d. real-valued random variables such that $\mathbb{E}[a_i] = \mu >0$ and
  $\mathbb{E}[(a_i^+)^2] < \infty$.
  Define $S_n =\sum_{i=1}^{n}a_i$ and for $a\leq 0 \leq b$ let $T$ be the stopping time
  $T = \inf\{ n \geq 1\mid S_n \notin [a,b] \}$. Then
  $$\mathbb{E}[T] \leq \frac{(1-\alpha)b + \alpha a}{\mu} + \frac{\mathbb{E}[(a_1^+)^2]}{\mu^2}$$
  where
  $\alpha = \mathbb{P}(S_T < a)$.
\end{prop}

\begin{lem}\label{funlemma}
  Let
  $\{a_i\}_{i\geq 1}$
  be a sequence of i.i.d. real-valued random variables and for $n\geq 1$ define
   the partial sums
   $S_n = \sum_{i=1}^{n}a_i$.
   Let 
   $M(r) = \mathbb{E}\left[\exp( r a_1 )\right]$
   be the moment generating function of the $a_i$
   and suppose there is a  $q>0$  such that $M(q) \leq 1$.
   Then for any $h\geq 0$,
  $$
  \mathbb{P}\left( \sup_{n\geq 1}S_n > h\right) \leq \exp(-q h).
  $$
\end{lem}
\begin{proof}
  For $n\geq 1$ define $Z_n = \exp(q S_n)$. Then for $n\geq 1$,
      \begin{align*}
      \mathbb{E}[Z_{n+1} \mid Z_n]
      &= \mathbb{E}\left[\prod_{i=1}^{n+1}\exp(q a_i) \, \Big| \,  Z_n\right] \\
      &= Z_n \mathbb{E}\left[\exp(q a_{n+1})\right] \\
      &= Z_n M(q)
      \leq Z_n.
      \end{align*}
      Furthermore,  for all
      $n\geq 1$
      it holds that
      $Z_n \geq 0$ and
      $\mathbb{E}[|Z_n|] = M(q)^n \leq 1$.
      Therefore
      $Z_n$ is a non-negative supermartingale. Hence
      \begin{align*}
        \mathbb{P}\left(\sup_{n\geq 1}S_n > h\right)
        &=
        \mathbb{P}\left(\sup_{n\geq 1}Z_n > \exp(qh) \right)
        \\
        &\leq \mathbb{E}[\exp( q a_1)]\exp(- qh) \\
        &\leq \exp(- q h).
      \end{align*}
      The first step in the above derivation follows from the monotonicity of the function
      $x \mapsto \exp(x)$,
       the second step follows from Theorem 7.8 in \citep{gallager1995discrete}, and the final step follows from our assumption on $q$.
\end{proof}

\subsubsection*{Proof of Proposition \ref{cusum-prop}}
  Define  $S_1,S_2,\hdots$ as
  $$S_n = \sum\limits_{i=1}^{n}\log\frac{f_1(x_i)}{f_0(x_i)}.$$
  and let $\alpha = \mathbb{P}_{\infty}( \sup_{n\geq 1}S_n > h )$.
It follows from Theorem 2 of \cite{lorden1971} that the CUSUM obeys
$$\operatorname{ARL2FA}_{\mathrm{CUSUM}} \geq \frac{1}{\alpha}.$$
Note that
$\mathbb{E}_{\infty}[\log \frac{f_1(x_i)}{f_0(x_i)}] = -d_{KL}(p_0,p_1) < 0$.
Hence $S_n$ is a random walk with negative drift.
Furthermore, the moment generating function of the increments under $\mathbb{P}_{\infty}$ can be expressed as 
$M(r) =
\mathbb{E}_{\infty}\left[
  \exp\left(
  r\log \frac{f_1(x_1)}{f_0(x_1)}
  \right)
  \right] =
\mathbb{E}_{\infty}
\left[
  \left(\frac{f_1(x_1)}{f_0(x_1)}\right)^{r}
\right]$
and it is evident that $g(1) = \mathbb{E}_{1}[1] = 1$. Then
we may apply Lemma  \ref{funlemma} with $q=1$ to conclude that $\alpha \leq \exp(-h)$. This establishes our claim on  the time to false alarm.

Let $N$ be the stopping time
$N = \inf\{ n \geq 1 \mid S_n > h\}$.
Again applying Theorem 2 from \cite{lorden1971}, it holds that
$$\operatorname{ESADD}_{\mathrm{CUSUM}} \leq \mathbb{E}_{1}[N].$$
Under the assumption of a change point at $t=1$, the variables
$\log \frac{f_1(x_i)}{f_0(x_i)}; i=1,2,\hdots$
are i.i.d with positive mean, and therefore $\{S_n\}_{n\geq 1}$ is a random walk with positive drift
$\mu = \mathbb{E}_{1}[\log \frac{f_1(x_1)}{f_0(x_1)}]$.
The bound on ESADD then follows from applying Proposition \ref{lordencor} to the stopping time $N$.

\hfill\ensuremath{\blacksquare}
\bibliographystyle{plain}
\bibliography{super.bib}
\end{document}